\documentclass[preprint]{imsart}

\usepackage{natbib}
\usepackage{hyperref} 
\usepackage{amsmath}
\usepackage{graphicx}
\usepackage{mathrsfs}
\usepackage{amsthm}
\usepackage{amscd}
\usepackage{amssymb}
\usepackage{latexsym}
\usepackage{stmaryrd}
\usepackage{wasysym}
\usepackage{enumitem}
\usepackage{verbatim}  
\usepackage{nicefrac}
\usepackage{algorithm}
\usepackage[noend]{algpseudocode}
\usepackage[disable]{todonotes}

\newtheorem{thm}{Theorem}
\newtheorem{cor}{Corollary}
\newtheorem{lem}{Lemma}

\newtheorem{eg}{Example}

\newcommand{\SL}{\mathbb{SL}} 
\newcommand{\SO}{\mathbb{SO}} 
\newcommand{\X}{\mathbf{X}} 
\newcommand{\G}{\mathbf{G}} 
\newcommand{\B}{\mathbf{B}} 
 
\newcommand{\R}{\mathbb{R}} 
\newcommand{\C}{\mathbb{C}} 
\newcommand{\HH}{\mathbb{H}_2} 

\newcommand{\technical}[1]{#1} 

\begin{document}
\begin{frontmatter}
\title{Lower Bounds for Kernel Density Estimation on Symmetric Spaces}
\runtitle{Lower Bounds for KDE}
\begin{aug}
\author{Dena Marie Asta}
\address{Department of Statistics\\ Ohio State University\\ Columbus, OH\\ 43210 USA\\
\href{Url}{dasta@stat.osu.edu}}
\runauthor{D. M. Asta}
\end{aug}

\begin{abstract}
  We prove that kernel density estimation on symmetric spaces of non-compact type, whose $L_2$-risk was bounded above in previous work \cite{Asta-gKDE2}, in fact achieves a minimax rate of convergence.
  With this result, the story for kernel density estimation on all symmetric spaces is completed.  
  The idea in adapting the proof for Euclidean space is to suitably abstract vector space operations on Euclidean space to both actions of symmetric groups and reparametrizations of Helgason-Fourier transforms and to use the fact that the exponential map for symmetric spaces of non-compact type defines a diffeomorphism.     
\end{abstract}

\end{frontmatter}




\section{Introduction}
Data, while often expressed as collections of real numbers, are often more naturally regarded as points in symmetric spaces, spaces that intuitively look the same from the vantage of any given point and include Euclidean space, spheres of various dimensions, the non-compact hyperboloid of constant negative curvature $\mathbb{H}_2$, and the space of symmetric $(3\times 3)$ positive definite matrices (eg. \citep{Rahman-et-al, Krioukov-et-al-hyperbolic-geometry, Asta-et-al}).
The literature offers some variants of kernel density estimation on smooth manifolds, including symmetric spaces.  
All symmetric spaces $X$ can be decomposed into symmetric spaces of \technical{Euclidean}, \technical{compact}, and \technical{non-compact} type in such a way that a KDE on $X$ can be constructed from KDEs on the three types.   
Symmetric spaces of the first two types admit KDEs with minimax convergence rates for $L_2$-risk given in terms of Sobolev constraints based on work by others (e.g. \cite{Pelletier}). 
A previous paper \cite{Asta-gKDE2} gives an upper bound on the convergence rate for $L_2$-risk of a KDE on symmetric spaces of the last type in terms of an order $\alpha$ Sobolev constraint (for any real $\alpha>0$).
This note completes the picture by showing that this upper bound is also a lower bound and therefore a minimax rate.

\begin{thm}
  \label{thm:lb}
  Consider the following data.
  \begin{enumerate}
    \item positive real number $Q>0$
    \item positive real $\alpha>0$
    \item symmetric space $\X$ of non-compact type
    \item density $f$ on $\X$ such that $\|\Delta^{\alpha/2}f\|_2\leq Q$.  
  \end{enumerate}
  Then there exists a constant $K>0$ such that $\inf_{\hat{f}^n}\mathbb{E}_f[(\hat{f}^n-f)^2]\geq Kn^{-\nicefrac{2\alpha}{2\alpha+\dim\,\X}}$, where the infimum is taken over all estimators $\hat{f}^n$ of $f$ based on $n$ samples drawn independently from $f$.
\end{thm}

The strategy, similar to a strategy used to derive lower bounds in the Euclidean case, is to adopt a standard and general scheme: lower bound the rate in terms of the mesh length with respect to some finite mesh in a Sobolev ball of densities of order $\alpha$. 
This strategy correctly gives a well-known lower bound for Euclidean space $\mathbb{R}^n$, although it is difficult to find a proof beyond the case $n=1$ in the literature.  
One hurdle in the non-Euclidean setting is that our constructions, such as constructions of kernels with prescribed bandwidths, need to be constructed at a sufficient level of generality.  
This means concretely that formulas involving subtraction and quotients need to be generalized to actions of symmetric groups and operations on \textit{Helgason-Fourier transforms}.  
Another hurdle in the non-Euclidean setting is the need to establish the existence of a suitable mesh in the symmetric space with the desired mesh length.  
This means concretely that we translate Euclidean meshes into meshes on a manifold along the \textit{exponential map} and exploit properties of such a map in the special case where the manifold is a symmetric space of non-compact type.  
The immediate corollary, based on our earlier work on upper bounds, is that the KDE $\hat{f}_{X_1,\ldots,X_n;h,T}$ on symmetric spaces of non-compact type based on a bandwidth $h$ and a certain cutoff parameter $T$ achieves the minimax rate.  

\begin{cor}
  Consider the following data.
  \begin{enumerate}
    \item positive real $\alpha>0$
    \item positive real number $Q>0$
    \item symmetric space $\X$ of non-compact type
  \end{enumerate}
  Then $\mathbb{E}_f\|\hat{f}_{X_1,\ldots,X_n;h_n,T_n}-f\|_2\asymp n^{-\nicefrac{2\alpha}{2\alpha+\dim\,\X}}$ achieves the minimax rate of convergence for an estimator of an $L_2$-density $f$ on $\X$ satisfying $\|\Delta^{\alpha/2}f\|_2\leq Q$ for some suitable choice of $h_n\rightarrow 0$ and $T_n\rightarrow\infty$.  
\end{cor}

We do not go into details in this note about the kernel density estimator $\hat{f}_{X_1,\ldots,X_n;h,T}$, instead referring the reader to our earlier work \cite{Asta-gKDE2} for definitions, examples, and simulations.  

\section{Background}
\subsection{Assouad's Lemma}
Fano's Lemma is a well-worn and general tool for deducing lower bounds on convergence rates for estimators.  
In fact, this method is used to establish the best possible lower bound for deconvolution kernel density estimators on Euclidean space as well as the negatively curved hyperboloid $\HH$ \citep{Huckemann-et-al}.  
However, this method relies on bounding the $\chi^2$-distance between sampling densities.  
Bounding such $\chi^2$-distances in practice is easiest when working under well-understood coordinate systems on $\R^n$ and $\HH$.  
It is difficult to get control over such $\chi^2$-distances in the general case.  
Instead, we use Assouad's Lemma, a variant of Fano's Lemma, defined in terms of \textit{Hamming distances}
$$H(u,v)=\sum_{i=1}^r|u_i-v_i|$$
between bit-vectors $u,v\in\{0,1\}^r$ and a well-defined notion of variation 
$$\|P\wedge Q\|=\int\min(p,q)\,\mu.$$
for probability measures $P,Q$ on a set $M$ having Radon-Nikodym derivatives $p$ and $q$ with respect to a measure $\mu$ on $M$.

\begin{lem}[Assouad's Lemma, {\cite[Theorem 24.3]{van-der-Vaart-asymptotic-stats}}]
  Consider the following data.
  \begin{enumerate}
    \item integer $r>0$
    \item a family $\{P_{\theta}\}_{\theta\in\{0,1\}^r}$ of probability measures on a common set indexed by elements in $\{0,1\}^r$
    \item a statistic $\psi$ on probability measures in the above family, as a function of $\theta\in\{0,1\}^r$, 
    \item an estimator $T$ of the statistic based on an observation from a model in the family $\{P_\theta\}_{\theta\in\{0,1\}^r}$
  \end{enumerate}
  Then $\max_{\theta}2^pE_\theta[d^p(T,\psi(\theta))]\geq\left(\min_{H(\theta,\theta')\geq 1}\frac{d^p(\psi(\theta),\psi(\theta'))}{H(\theta,\theta')}\frac{r}{2}\right)\left(\min_{H(\theta,\theta')=1}\|P_\theta\wedge P_{\theta'}\|\right)$.
\end{lem}

The application to density estimation comes from taking all the probability measures in the lemma to describe the product probability measures based on $n$ independent samples.  
Moreover, the lemma is useful for us insofar as we can relate the variation term $\|P_\theta\wedge P_{\theta'}\|$ with integrals that look more like $L_2$-distances.  
For that reason, the following observation made for the case $\X=\R$ also applies to case of general $\X$.  
Write $H^2(P,Q)$ for the square of the \textit{Hellinger distance} between probability measures $P$ and $Q$ on a set $M$, the well-defined integral
$$H^2(P,Q)=\int_M(\sqrt{p}-\sqrt{q})\,\mu$$
for some choice of measure on $M$ with respect to which $P$ and $Q$ have Radon-Nikodym derivatives $p$ and $q$.  
The following lemma is standard and not original to this paper.

\begin{lem}
  Consider the following data.
  \begin{enumerate}
    \item measure space $M$ with measure $\mu$
    \item absolutely continuous probability measures $P$ and $Q$ on $M$
  \end{enumerate}
  Then $\|P^n\wedge Q^n\|\geq\frac{1}{2}\left(1-\frac{1}{2}H^2(P,Q)\right)^{2n}$.  
\end{lem}

\subsection{Symmetric Spaces}
A symmetric space is a special type of smooth (Riemannian) manifold that looks the same at every vantage point.
In fact, a Riemannian manifold having no 1-dimensional holes and forming a complete metric space is a symmetric space exactly when its curvature is constant.  
We refer the reader to the previous paper \cite{Asta-gKDE2} for the relevant theory. 
\textit{Lie groups}, manifolds with compatible smooth invertible and associative multiplications and smooth inversion operations, are examples of symmetric spaces; examples to keep in mind are smooth manifolds of invertible matrices closed under matrix multiplication and matrix inversion.    
A general symmetric space need not form a Lie group, but symmetric spaces are always constructed from Lie groups of their symmetries.
An example to keep in mind is the \technical{Poincar\'{e} halfplane}.

\begin{eg}
  \label{eg:H2}
  The \technical{Poincar\'{e} halfplane} $\HH$ is the $2$-manifold defined as the subspace
  $$\HH=\{z\in\mathbb{C}\;|\;\mathrm{Im}(z)>0\},$$
  of $\C$ equipped with the Riemannian metric given by the arc length
  $$ds^2=(\mathrm{Im}\;z)^{-2}(d(\mathrm{Re}\;z)^2+d(\mathrm{Im}\;z)^2).$$
  This space can be interpreted as the information manifold of all univariate normal distributions, where the real coordinates describe the means, and the imaginary coordinates describe standard deviations, and the Riemannian metric is the Fisher metric.  
  Alternatively, this space is a natural latent space for families of random graphs used to model real-world networks \cite{Krioukov-et-al-hyperbolic-geometry}. 
  Alternatively, this space models electrical impedances on which certain circuit elements act as M\"{o}bius transformations \citep{Huckemann-et-al}.
\end{eg}

Every symmetric space can be thus expressed as a certain quotient of a Lie group by a Lie subgroup.  
Write $\G$ for a \textit{semisimple Lie group with finite center} and $\mathbf{H}$ for its maximal compact Lie subgroup..  
Then $\X=\G/\mathbf{H}$ is a \textit{symmetric space of noncompact type} and in fact all such symmetric spaces arise in this manner.  
The Lie group $\G$ is a group of isometries on $\X$: each element $g\in\G$ determines an isometry $\X\cong \X$ defined by multiplication on $\G$.  

\begin{eg}
  \label{eg:mobius}
  The space $\SL_2$ of $(2\times 2)$-matrices with determinant $1$ acts on $\HH$ by \technical{M\"{o}bius transformations}:
  $$\left( \begin{array}{cc} a & b \\ c & d \end{array} \right)(z)=\frac{az+b}{cz+d}.$$
  The matrices in $\SL_2$ fixing $i\in\HH$ form the matrix subgroup $\SO_2$ of $(2\times 2))$ rotation matrices.
  The action of $\SL_2$ on $\HH$ implicitly gives a well-defined bijection $\SL_2/\SO_2\cong\HH$ sending an equivalence class of a matrix $m\in\SL_2$ to $m(i)\in\HH$.
  This bijection $\HH\cong\SL_2/\SO_2$ defines an isometry for a suitable choice of bi-$\SO_2$-invariant inner product on the Lie algebra $\mathfrak{sl}_2$ associated to $\SL_2$ \cite[\S3.1]{Terras-harmonic-1}.
  Thus $\HH$ is a Riemmanian symmetric space.  
\end{eg}

We write $\rho_{\X}$ for half of the sum of the \textit{positive root weights} of $\G$.  

\begin{eg}
  \label{eg:root.weight}
  For $\X=\HH$, $\rho_{\X}=\rho_{\HH}=\nicefrac{1}{2}$. 
\end{eg}

We will write ${\bf 1}$ for the point in the quotient $\X=\G/{\bf H}$ represented by the identity element in $\G$.
The \textit{exponential map} is the smooth map $\mathrm{exp}:\mathbb{R}^{\dim\,X}\rightarrow\X$ sending each point $v\in\mathbb{R}^n$ to the point $\gamma_v(1)\in\X$, where $\gamma$ is the unique geodesic $[0,1]\rightarrow M$ such that $\gamma(0)={\bf 1}$ and whose derivative $\gamma'(0)$ at $0$ is $v$.  
We will use the fact that because $\X$ is symmetric of non-compact type, the exponential map defines a diffeomorphism
$$\mathrm{exp}:\mathbb{R}^{\dim\,\X}\cong\X.$$
that increases distances in the sense that $\|u-v\|_2$ is a lower bound for the Riemannian distance between $\mathrm{exp}(u)$ and $\mathrm{exp}(v)$ for all pairs $u,v\in\mathbb{R}^{\dim\,\X}$.  

\subsection{Helgason-Fourier Analysis}
Helgason-Fourier Analysis is an analogue of Fourier Analysis for symmetric spaces of non-compact type.  
The reader is referred to \citep[Section 2]{Pesenson} for a concise summary of the theory and \citep{Terras-harmonic-1} for details in the special case $\X=\HH$.  
Write $\mathcal{H}[f]$ for the \textit{Helgason-Fourier transform} of an $L_2$-function $f:\X\rightarrow\mathbb{C}$.  
This transform $\mathcal{H}$ defines a linear isometry from $L_2(\X)$ to an $L_2$-space of complex-valued functions on a certain \textit{frequency space} that depends on $\X$.  
Unlike with the ordinary Fourier transform for $\mathbb{R}^n$, this frequency space is not also $\X$ but instead a product of Euclidean space with another space.
The exact definition of this frequency space depends on an algebraic decomposition of $\G$ and requires too many preliminaries, and so we refer the reader to \citep[Section 2]{Pesenson} or the previous paper \cite{Asta-gKDE2} for its definition.  
Adopting the same notation as in our previous paper \cite{Asta-gKDE2}, we write $\mathfrak{a}\times\B$ for this frequency product space, where ${\bf \mathfrak{a}}$ is a finite dimensional real vector space and $\B$ is a certain quotient of $\mathbb{H}$.  
We have a Plancherel Theorem, the observation that the linear map $\mathcal{H}$ is in fact a linear isometry of $L_2$-spaces:
$$\mathcal{H}:L_2(\X)\cong L_2(\mathfrak{a}\times\B).$$
One thing to note is that the measure on $\mathfrak{a}\times\B$ giving us our Plancherel Theorem, is not at all straightforward to define, and involves a \textit{Harish-Chandra c-function} central to a great deal of Geometric Analysis.
We will omit notation for the measure on $\mathfrak{a}\times\B$ and the volume measure on $\X$ when integrating functions over them.  

\begin{eg}
  \label{eg:transform}
  Take $\X=\HH$. 
  The factors $\mathfrak{a}$ and $\mathbf{B}$ of frequency space are given by
  $$\mathfrak{a}=\mathbb{R}\quad\mathbf{B}=\SO_2/\mathbb{Z}_2,$$
  where $\mathbb{Z}_2$ is the discrete subset of the identity rotation and the rotation by $180$ degrees.  
  Then
  $$\mathcal{H}[f](\lambda,[k])=\int_{\HH}f(z)\,\mathrm{Im}(k(z))^{\frac{1}{2}-\lambda}\,dz,$$
  where $[k]$ denotes the equivalence class in $\B$ represented by $k\in\mathbf{H}$.  
\end{eg}

The Helgason-Fourier transform sends convolutions to products in a suitable sense.  
The previous paper from the author analyzing a KDE \cite{Asta-gKDE2} or other work deconvolving noise \cite{chevallier2017kernel,Huckemann-et-al} uses this property to obtain convergence rate bounds.  
Since the constructions in this paper do not require convolutions, we do not go into the details and refer the reader to \citep[Section 2]{Pesenson} for details.  
One property that is crucial for us is the ability to use Helgason-Fourier transforms, like ordinary Fourier transforms, to define $\alpha$-order derivatives even when $\alpha$ is not a positive integer.  
Recall that the ordinary Fourier transform $\mathcal{F}$ satisfies the property
\begin{equation}
  \label{eqn:ordinary.derivatives}
  \mathcal{F}[f^{(\alpha)}](s)=i\alpha\mathcal{F}[f](\alpha s)
\end{equation}
for all positive integers $\alpha$ and functions $f\in L_2(\mathbb{R})$ having all orders of derivatives up to $\alpha$.  
We can then extend the definition of $f^{(\alpha)}$ for all positive reals $\alpha$ by taking (\ref{eqn:ordinary.derivatives}) as a definition.
Since the Laplace-Beltrami operator $\Delta$ on a Riemannian manifold is a second-order differential operator, we write $\Delta^{\alpha/2}f$ for our analogue of $\alpha$-order derivatives.  
For each $f\in L_2(\X)$, define $\Delta^{\alpha/2}f\in L_2(\X)$ by
$$\mathcal{H}[\Delta^{\alpha/2}f](\lambda,b)=(\rho_{\X}^2-\lambda^2)^{\alpha/2}\mathcal{H}[f](\lambda,b),$$

The \textit{order $\alpha$ Sobolev ball $\mathcal{F}_\alpha(Q)$  of radius $Q$} in $L_2(\X)$ can then be defined as
$$\mathcal{F}_\alpha(Q)=\{f\in L_2(\X)\;|\;\|\Delta^{\alpha/2}f\|_2\leqslant Q\}.$$

For each smooth $L_2$-function $K:\X\rightarrow\mathbb{R}$ and $h>0$, write $K_h$ for the smooth $L_2$-function $\X\rightarrow\mathbb{R}$ that is characterized in terms of its Helgason-Fourier transform as follows:
$$\mathcal{H}[K_h](\lambda,b)=\mathcal{H}[K](h\lambda,b).$$

\section{Main Proof}
Let $d$ be the dimension $\dim\,\X$ of the symmetric space $\X$.    
Define
$$h_n=n^{-\nicefrac{1}{2\alpha+d}}\quad r_n=\lfloor n^{\nicefrac{d}{2\alpha+d}}\rfloor.$$

Fix a smooth $L_2$-density $f$ on $\X$.
Fix a smooth real-valued bump function $K$ on $\X$ with support given in the open disk in $\X$ about the point ${\bf 1}$ of radius $1$.  
Let $\theta,\theta'$ denote elements in $\{0,1\}^{r_n}$.  
There exists a mesh $X_{n,1},\ldots,X_{n,r_n}$ in $\R^d$ whose mesh width is $h_n$, by evenly subdividing a $d$-dimensional cube of edge length $1$ in $\R^d$ into smaller cubes each of which has edge length $h_n$. 
It therefore follows that $\mathrm{exp}(X_{n,1}),\ldots,\mathrm{exp}(X_{n,r_n^d})$ gives a mesh in $\X$ whose mesh width is at least $h_n$.  
Let $g_{x}$ denote a choice of isometry $\X\cong\X$ in $\G$ sending $x$ to ${\bf 1}$ and continuous in $x$.
For each $n=1,2,\ldots$ and $\theta$, define:

$$f_{n,\theta}=f+h_n^{\alpha}\sum_{j=1}^{r_n}\theta_j K_{h_n^d}g_{X_{nj}}.$$

\begin{lem}
  \label{lem:smoothness}
  Fix $n$ and $Q>0$.
  For both $\|\Delta^{\alpha/2}f\|_2$ and $\|\Delta^{\alpha/2}K\|_2$ sufficiently small, 
  $$f_{n,\theta}\in\mathcal{F}_\alpha(Q).$$
\end{lem}
\begin{proof}
Note that we can bound $\|\Delta^{\alpha/2}K_{h_n}\|_2$ by the inequalities
\begin{align*}
	\|\Delta^{\alpha/2}K_{h_n^d}\|_2^2 &= \|\mathcal{H}[\Delta^{\alpha/2}K_{h_n^d}]\|_2^2\\
	&=\int_{\X}(\rho_{\X}^2-\lambda^2)^{\alpha}\mathcal{H}[K](h_n^d\lambda,b)^2\\
	&\leq Ch_n^{d-2\alpha d}\int_{\X}(\rho_{\X}^2-u^2)^{\alpha}\mathcal{H}[K](u,b)^2=Cn^{\nicefrac{d(2\alpha-1)}{2\alpha+d}}\|\Delta^{\alpha/2}K\|^2_2
\end{align*}
for some constant $C>0$.  
Therefore we can deduce that
\begin{align*}
	\|\Delta^{\alpha/2}f_{n,\theta}\|_2 & \leq\|\Delta^{\alpha/2}f\|_2+n^{-\nicefrac{\alpha}{2\alpha+d}}\sum_{\theta_j=1}\|\Delta^{\alpha/2}K_{h_n^d}g_{X_{nj}}\|_2\\
	& \leq\|\Delta^{\alpha/2}f\|_2+n^{-\nicefrac{\alpha}{2\alpha+d}}\sum_{\theta_j=1}\|\Delta^{\alpha/2}K_{h_n^d}\|_2\\
	&\leq \|\Delta^{\alpha/2}f\|_2+n^{\nicefrac{d-\alpha}{2\alpha+d}}\|\Delta^{\alpha/2}K_{h_n^d}\|_2\\
	&\leq \|\Delta^{\alpha/2}f\|_2+Cn^{\nicefrac{d-\alpha+d(2\alpha-1)}{2\alpha+d}}\|\Delta^{\alpha/2}K\|_2
\end{align*}
  For sufficiently small  $\|\Delta^{\alpha/2}f\|_2$ and $\|\Delta^{\alpha/2}K\|_2$ but fixed $n$, this quantity can be made arbitrarily small.
\end{proof}

\begin{lem}
  \label{lem:bounds}
  There exists a constant $C>0$ such that for $\theta,\theta'\in\{0,1\}^{r_n}$, 
  $$\|f_{n,\theta}-f_{n,\theta'}\|_2^2=Cn^{-1}H(\theta,\theta')$$
  and if $H(\theta,\theta')=1$, $\|P_{\theta}\wedge P_{\theta'}\|\geq (1-\mathcal{O}(n^{-1}))^{2n}$.
\end{lem}
\begin{proof}
  Note that
\begin{align*}
	H^2(P_{\theta},P_{\theta'}) 
	&= \int_{\X}(\sqrt{f_{n,\theta}}-\sqrt{f_{n,\theta'}})^2\,dx\\
	&= \int_{\X}\left(\frac{(f_{n,\theta}-f_{n,\theta'})}{\sqrt{f_{n,\theta}}+\sqrt{f_{n,\theta'}}}\right)^2\\
	&\geq C_1\|f_{n,\theta}-f_{n,\theta'}\|_2^2\\
	&= C_1h^{2\alpha}_n\sum_{j=1}^{r_n^d}|\theta_j-\theta_j'|^2\int_{\X}K^2_{h_n^d}(g_{\mathrm{exp}(X_{nj})})\\
	&= C_1h^{2\alpha}_n\sum_{j=1}^{r_n^d}|\theta_j-\theta_j'|^2\int_{\X}K^2_{h_n^d}\\
	&= C_1h^{2\alpha}_n\sum_{j=1}^{r_n^d}|\theta_j-\theta_j'|^2\int_{\mathfrak{a}\times{\mathbf B}}\mathcal{H}[K](h_n^d\lambda,b)\\
	&= C_1h^{2\alpha+d}_nH(\theta,\theta')\|\mathcal{H}[K]\|_2^2=C_2n^{-1}H(\theta,\theta')
\end{align*}
for some constants $C_1,C_2>0$.  
Thus when $H(\theta,\theta')=1$,
$$\|P_{\theta}^n\wedge P_{\theta'}^n\|\geq\left(1-\frac{1}{2}H^2(P_{\theta},P_{\theta'})\right)^{2n}\geq(1-\mathcal{O}(n^{-1}))^{2n}.$$
\end{proof}

\begin{proof}[proof of Theorem \ref{thm:lb}]
  Lemma \ref{lem:smoothness} implies that we can apply Assouad's Lemma to the case
  $$r=r_n,\,\psi(\theta)=f_{n,\theta},\,T=\hat{f}^n,\,P_{\theta}=P_{n,\theta}^n,$$
  where $P_{n,\theta}^n$ denotes the probability distribution on $(X_1,\ldots,X_n)$ for $X_1,\ldots,X_n$ independently drawn from $f_{n,\theta}$, to get that $\max_{\theta}2^pE_\theta\|\hat{f}_n-f_{n,\theta}\|_2^2\,dx\geq C_2h_n^{2\alpha+\dim\,X}n^{-1}\frac{r_n^d}{2}(1-O(n^{-1}))^{2n}$.
  Thus we see that
  \begin{align*}
  \inf_{\hat{f}^n}\mathbb{E}_f[(\hat{f}^n-f)^2] 
  &\geq\max_{\theta}2^pE_\theta\|\hat{f}_n-f_{n,\theta}\|_2^2 \\
  &\geq C_2h_n^{2\alpha+d}\frac{r_n}{2}(1-O(n^{-1}))^{2n}\\
  &\geq \mathcal{O}(n^{-\nicefrac{2\alpha}{2\alpha+d}}).
\end{align*}
where the second line follows from Lemma \ref{lem:bounds}.
\end{proof}

\bibliographystyle{imsart-nameyear}
\bibliography{bib_asta}






\end{document}